\documentclass[12pt]{amsart}
\usepackage{graphicx}
\usepackage{amsmath}
\usepackage{amssymb}
\usepackage{setspace}
\newtheorem{thm}{Theorem}[section]

\newtheorem{lem}[thm]{Lemma}
\newtheorem{prop}[thm]{Proposition}
\theoremstyle{definition}

\theoremstyle{remark}
\newtheorem{rem}[thm]{Remark}
\numberwithin{equation}{section}
\begin{document}
%%%%%%%%%%%%%%%%%%%%%%%%%%%%%%%%%%The title and abstract

\title[Minimal action odd solutions]{ Existence and axial symmetry of minimal odd solutions for 2-D Schr\"{o}dinger-Newton equation }
\author{Yang Zhang}
\address{School of Mathematics and Statistics, Central South University, Changsha 410075, People’s Republic of China}
\email{zhangyang@amss.ac.cn}
%%%%%%%%%%%%%%%%%%%%%%%%%%%%%%%%%%%%%%%%%%%%%%%%%%%%%%%%%%%%%%%%%%The abstract
\begin{abstract}
We consider the following 2-D Schr\"{o}dinger-Newton equation  
\begin{eqnarray*}
\begin{cases}
-\Delta u+u=w|u|^{p-1}u  \\ 
-\Delta w=2 \pi |u|^p
\end{cases}\text{in} \; \mathbb{R}^2 
\end{eqnarray*}
for $ p \geq 2 $. Using variational method with the Cerami compactness property, we prove the existence of minimal action odd solutions. Also by carefully applying the method of
moving plane   to a similar but more complex equation on the upper half space, we prove these solutions are in fact axially symmetric. Our results partially can be seen
as the counterpart of results in paper \cite{GS} for the 2-D case, or the extension of the results \cite{CW} to the odd solution case.
\end{abstract}
%%%%%%%%%%%%%%%%%%%%%%%%%%%%%%%%%%%%%%%%%%%%%%%%%%%%%%%%%%%%%%
\maketitle

%%%%%%%%%%%%%%%%%%%%%%%%%%%%%%%%%%%%%%%%%%%%%%%%%%%%%%%%%%%%%%%%%%The key words
{\small {\bf Keywords:} Logarithmic convolution potential, Cerami compactness, Method of moving plane, Schr\"{o}dinger-Newton equation.

%%%%%%%%%%%%%%%%%%%%% Below is the Introuction section
\section{Introduction}
In this paper, we study the following  Schr\"{o}dinger-Newton equation in $\mathbb{R}^2 $
\begin{eqnarray}\label{eq1.1}
\begin{cases}
-\Delta u+u=w|u|^{p-1}u  \\ 
-\Delta w=2 \pi |u|^p
\end{cases}\text{in} \; \mathbb{R}^2    
\end{eqnarray}
for $  p \geq 2.$ It arises in many different physical models, see \cite{Pekar,P}. Depending on these different models, it could be given different names: Choquard or
Choquard-Pekar equation, Schr\"{o}dinger-Newton equation, or stationary Hartree equation. Here we call it Schr\"{o}dinger-Newton equation, based on the Penrose's 
model for Newton gravitation counpled with the quantum physics. In 3-D case, this equation has been widely studied, where the fundamental solution for $-\Delta$ is  $\frac{1}{|x|}$, a special Riesz potential of order $2$. First Lieb \cite{L} poved the  the existence and
uniquness of radial positive ground state solutions for $p=2$. Then Lions proved there are infinitely many radial solutions, see \cite{Lions}. For genneral $p$ and other results, see \cite{Acker1,Acker2,BJ,GS,MS2,R}. For the complete mathematical results,
we strongly recommand the impressive survey \cite{MS1} and the references therein, where the authors also listed many interesting open problems.
\par
In 2-D case, the analysis for this equation is harder, because of the sign-changing property of the $\log$ function, which is the fundamental solution of $\Delta$ in $\mathbb{R}^2$. First Choquard, Stubby and Vuffray proved there is a unique radial ground state solution by an ODE matheod for $p =2$, see \cite{CSV}. Then Stubby
established the variational framework and proved a stronger result using constraint minimization argument, see \cite{Stubbe}. Based on this variational framework, Cingolani and Weth \cite{CW}
discovered the energy functional or aciton functional $(p=2)$  satisfies the so-called Cerami compactness property, and used the minimax procedure to give the variational characterization 
of the ground state solution. They also show the symmetry of these solutions and other properties. Additionally they proved the  existence of infinitely many solutions of which the energies go to infinity and have many different types of 
symmetry in terms of group $G$, see also \cite{DW}. Later Cao, Dai and Zhang extended these resuts to the general $p \geq 2$ using the same method,
see \cite{DZ}. For the sharp decay
and non-degenerency, see \cite{BCS}.
\par
In \cite{GS}, the authors considered  the existence of the minimal action odd solutions and minimal action nodal solutions in $\mathbb{R}^3$. From the results in \cite{CW,DZ} for
 2-D case, we know there indeed exist odd solutions and nodal solutions. So the odd solutions set and nodal solutions set are not empty. The natural questions for us are 
 whether there is a minimal odd solution among all the odd solutions, and whether these minimal action odd solutions are axially symmetric. We will give these two questions a 
 firmative answer. Our results can be seen as the counterpart of \cite{GS} for the 2-D case, or can be seen as the extension of \cite{CW,DZ} to the odd solutions case.
\par
We consider the energy functional or actional functional by
$$ I(u)=\frac{1}{2} \int_{\mathbb{R}^2} \big( |\nabla u|^2  +   u^2 \big)dx +  \frac{1}{2p} \iint \log{|x-y|}|u|^p(x)|u|^p(y) dxdy   $$
defined on the functions space
$$ X=\left \{ u\in H^1(\mathbb{R}^2):\int_{\mathbb{R}^2}\log{(1+|x|)} |u|^p(x)dx < \infty        \right \}
:=H^1 \cap L^p(d \mu) ,$$
where the Radon measure is $ d \mu =\log{(1+|x|)} dx$. Formally,  the Schr\"{o}dinger-Newton equation is the corresponding Euler-Lagrange equation for this energy functional.
The properties of the actional functional and function space $X$ will be given below, see also \cite{CW,DZ}.
\par
Now, we define the odd function space
\begin{align*}
 X_{odd} &:=\bigg \{   u \in X \bigg | u(x_1,-x_2)=-u(x_1,x_2) \;\text{for almost every }\; x=(x_1,x_2) \in \mathbb{R}^2   \bigg \} \\
         &= H^1_{odd} \cap L^p(d \mu).
\end{align*}
The norm on $X_{odd}$ is defined by 
$$\|u\|_{X}:=\|u\|_{H^1} + \|u\|_{L^p(d \mu)}=\|u\|_{H^1} + \|u\|_*.$$
The odd Nehari manifold and corresponding minimum is defined by
\begin{align*}
\mathcal{N}_{odd}&:= \bigg \{  u \in X_{odd}: \left \langle  I^{\prime}(u),u  \right\rangle =0 , u\not \equiv 0 \bigg\}\\
       &=\mathcal{N}  \cap X_{odd},\\
c_{odd}&:= \inf\limits_{\mathcal{N}_{odd}} I(u),
\end{align*}
where $\mathcal{N} $ is the Nehari manfold $ \mathcal{N}= \big \{  u \in X: \left \langle  I^{\prime}(u),u  \right\rangle =0 , u\not \equiv 0 \big\} $.
Note that $\mathcal{N}$ is not empty, since we can always choose $u$ with $  \iint \log{|x-y|}|u|^p(x)|u|^p(y) dxdy <0 $, such that 
$\langle I^{\prime}(tu),tu  \rangle =0 $ for some $t>0$.
Also we define the odd ground state value by 
$$ c_{g,odd}:= \inf \bigg \{ I(u) : u \not\equiv 0, u \in X_{odd},I^{\prime}(u)=0  \bigg\} .$$
We shall call $c_{g,odd}$ 
the minimal action value, and the corresponding solutions are the minimal action odd solutions, if they exist.
The first minimax value is regularly defined on the function space $X_{odd}$ by
$$ c_{mm,odd}:=\inf\limits_{u \not \equiv 0} \sup\limits_{t >0} I(tu) .$$
Also the mountain pass value is  defined by
$$c_{mp,odd}:= \inf\limits_{\gamma \in \Gamma} \sup\limits_{t \in [0,1]} I \circ \gamma(t),$$
where $ \Gamma = \left\{  \gamma \in C([0,1],X_{odd}) : \gamma(0)=0,I\circ\gamma(1)<0 \right\} .$
\par
Our first result is the existence of minimal action odd solutions:
\begin{thm}
	Assume $ p\geq 2 .$ Then we have:
\begin{enumerate}
	\item $c_{mp,odd}>0$
	\item there exists an odd solution $u \in X_{odd}\setminus \{0\} ,$ such that $I(u)=c_{mp,odd}$;
	\item $ c_{g,odd}=c_{odd}=c_{mm,odd}=c_{mp,odd}; $
	\item $c_{g,odd}> c_g$ strictly, where $c_g$ is the ground state energy in $X$.
\end{enumerate}
\end{thm}
Our second result is the axial symmetry for all the minimal action odd solutions.
\begin{thm}
	If $u$ is a minimal action odd solution, then $u$ is positive or negative on the upper halfplane $ \mathbb{R}^2_{+} 
	=\left\{x=(x_1,x_2):x_2 >0\right\}.$ Moreover, $u$ is axially symmetric with respect to some axis perpendicular to $ \partial \mathbb{R}^2_+, $
	and $\frac{\partial u}{\partial x_1}<0 $ along the each array starting from the axis in $x_1$ direction.
\end{thm}
The proof for this axial symmetry property is based on the method of moving plane. For this robust method, see \cite{CLO,DQ,GNN,LiZhu,MZ,WX}.
%%%%%%%%%%%%%%%%%%%%%%%%%%%%%%%%%%%%%%%%%%%%%%%%%%%%%%%%%%%%%%%%%%%%%%%%%%%%%%%%%%%%%%%%%%%%%%%%%%%%%%%%%%%%%%%%%%%%%%%%%%%%%%%%%%%%%%%%%%%%
%%%%%%%%%%%%%%%%%%%%%%%%%%%%%%%%%%%%%%%%%%%%%%%%%%%%%%%%%%%%%%%%%%%%%%%%%%%%%%%%%%%%%%%%%%%%%%%%%%%%%%%%%%%%%%%%%%%%%%%%%%%%%%%%%%%%%%%%%%%%%%%%%%
%%%%%%%%%%%%%%%%%%%%%%%%%%%%%%%%%%%%%%%%%%%%%%%%%%%%%%%%%%%%%%%%%%%%%%%%%%%%%%%%%%%%%%%%%%%%%%%%%%%%%%%%%%%%%%%%%%%%%%%%%%%%%%%%%%%%%%%%%%%%%%%%%%%%%%%%%%%%%%%%
\section{Preliminaries}
In this section, we list some preliminaries for proving this two theorems, see the details in \cite{CW} for $p=2$ and \cite{DZ} for $p \geq 2$. We start from an elementary
but very useful inequality, which it's first used in proving the famous Brezis-Lieb lemma.
\begin{lem}[$\epsilon$-inequality]
	Let $0<p<\infty $ be a fixed number. For each given $\epsilon >0 ,$ there is a $C_{\epsilon}>0$, such that for all $a,b \in \mathbb{C}$, we have
	$$ \bigg| |a+b|^p-|b|^p  \bigg| \leq \epsilon |b|^p  + C_{\epsilon}|a|^p .  $$
\end{lem}
The proof can be seen in \cite{LL}. We will apply this simple $\epsilon$-inequality in proving the strong convergence of Cerami sequence and the axial symmetry of minimal action odd solutions. We introduce the bilinear form by
\begin{align*}
B_1(f,g) &= \iint \log(1+|x-y|)f(x)g(y)dxdy,\\
B_2(f,g) &= \iint \log(1+\frac{1}{|x-y|})f(x)g(y)dxdy, \\
B_0(f,g)&= B_1(f,g)-B_2(f,g)=     \iint \log(|x-y|)f(x)g(y)dxdy
\end{align*}
and the corresponding functionals
\begin{align*}
 V_1(u) &=B_1(|u|^p,|u|^p)= \iint \log(1+|x-y|)|u|^p(x)|u|^p(y)dxdy, \\
V_2(u) &=B_2(|u|^p,|u|^p)= \iint \log(1+\frac{1}{|x-y|})|u|^p(x)|u|^p(y)dxdy, \\
V_0(u) &=B_0(|u|^p,|u|^p)= \iint \log(|x-y|)|u|^p(x)|u|^p(y)dxdy.
\end{align*}
By the HLS inequality, we can bound  $V_2(u)$ by:
$$  |V_2(u)| \leq C \||u|^p \|_{L^{\frac{4}{3}}}^2= C \|u \|_{L^{\frac{4p}{3}}}^{2p} .$$
Using these notations, we can rewrite the action functional in a compact form
\begin{align*}
 I(u) & =\frac{1}{2}\|u\|_{H^1}^2 + \frac{1}{2p}V_0(u) \\
      & = \frac{1}{2}\|u\|_{H^1}^2 + \frac{1}{2p} \left(  V_1(u)  - V_2(u)  \right) 
      \end{align*}
defined on the odd fucntion space $X_{odd}$ with the norm $\|u\|_X=\|u\|_{H^1}+\|u\|_* $.
\par
The next are the properties of the action functional and function space.
\begin{lem}
\begin{enumerate}
	\item The function space $X=H^1\cap L^p(d\mu) $  is compactly embedding in $L^s(\mathbb{R}^2)$ for all $s \in [p,\infty);$
	\item The functionals $V_1,V_2,V_0$ and $I$ is $C^1(X):$ for each $u, v $ in $X$, $i=0,1,2,$
	              $ \big \langle V_i^{\prime}(u),v \big \rangle = 2p B_i(|u|^p,|u|^{p-2}uv); $  
	\item    $V_1$ is weakly lower semicontiniuous on $H^1(\mathbb{R}^2)$;
	      $I$ is weakly lower semecontiniuous on $X$ and is lower semicontiniunous on $H^1.$         
\end{enumerate}
\end{lem}
\begin{proof}
We only prove property (1). 
We have already known  the embeddding $H^1(\mathbb{R}^2)\hookrightarrow L^r(\mathbb{R}^2)(dx)$ for $ 2\leq r <\infty $ is  locally compact by the Rellich-Kondrachov compactness theorem.
Now by the Kolmogrov-M.Reize-Frechet compactness criteria, see \cite{B}, we only need to check the uniformly integralbility  at infinity for
$(u_n) \subset X$ bounded.
Notice for each $\epsilon >0,$ choose $R$ large enough, then we have
$$ M \geq \int_{|x|\geq R} \log (1+|x|)|u_n|^p dx \geq \int_{|x|\geq R}\log (1+R)|u_n|^p dx ,$$
$$ \int_{|x| \geq R}|u_n|^p dx \leq \frac{M}{\log (1+R)} \leq \epsilon, \qquad \text{for all $R$ large enough}, $$
yielding the uniformly integralbility. Hence the embedding $X\hookrightarrow \hookrightarrow L^p $ is compact. By the Gagliardo-Nirenberg interpolation inequality,
for all $ s \in [p, \infty) $, the embedding is also compact.
\end{proof}
The properties of the solutions are listed in the following lemma.
\begin{lem}\label{lem2.3}
\begin{enumerate}
	\item If $u $ is the critical point of the energy functional, then $u$ is the weak solution of  the following Euler-Lagrange equation: 
$$  -\Delta u+u+ (\log |\cdot| * |u|^p) {|u|}^{p-2}u =0  .$$
	\item  The potential function defined by $ w(x):=\int_{\mathbb{R}^2} \log|x-y| |u|^p(y)dy $
	is of class $C^3$,  hence $ -\Delta w = 2\pi   |u|^p $ classically. Moreover, we have
	$ w(x)-  \log|x| \int_{\mathbb{R}^2} |u|^p \longrightarrow 0$, as $ x \to \infty $, and $ |\nabla w | \to 0  $ as $ x \to \infty ;$
	\item  $u$  decay  exponentially : for any $\epsilon >0$, there is a $C_{\epsilon}>0, $
	such that: $$|u(x)|\leq C_{\epsilon} \exp^{-(1-\epsilon)|x|}; $$
	\item  $ u $ is $ W^{2,r}(\mathbb{R}^2) $, for every $r \in (1,\infty)$,  hence is the strong solution of the Euler-Lagrange equation, in fact $ u \in C^{2,\alpha}_{loc}.$
\end{enumerate}
\end{lem}	
\begin{rem}
 By Lemma \ref{lem2.3} of property (2), if $u \in H^1_0(\mathbb{R}^2_+) \cap L^p(d \mu) \subset X, $ then we have
$$ w(x)-  \log|x| \int_{\mathbb{R}^2_+} |u|^p\longrightarrow 0, \qquad \text{as} \; x \longrightarrow \infty  .$$
Now for $u \in X_{odd}=H^1_{odd} \cap L^p(d \mu) ,$  we have another asymptotics:
$$ w(x)-  \log|x| \int_{\mathbb{R}^2} |u|^p \longrightarrow 0  .$$
But by the odd symmetry, we have 
\begin{align*}
w(x) &=\int_{\mathbb{R}^2} \log|x-y| |u|^p(y)dy \\
     &= \frac{1}{2} \int_{\mathbb{R}^2}  \bigg( \log \big[(x_1 -y_1)^2 +(x_2-y_2)^2\big]   \bigg) |u|^p(y)dy\\
     &= \frac{1}{2}   \int_{\mathbb{R}^2_+}  \bigg( \log   \big( |x -y|^2  \big) \bigg) |u|^p(y)dy  +  
         \frac{1}{2} \int_{\mathbb{R}^2_+}  \bigg( \log \big[(x_1 -y_1)^2 +(x_2+y_2)^2\big]   \bigg) |u|^p(y)dy.
\end{align*}
Combining these two asymptotics, we have
$$ \frac{1}{2} \int_{\mathbb{R}^2_+}  \bigg( \log \big[(x_1 -y_1)^2 +(x_2+y_2)^2\big]   \bigg) |u|^p(y)dy  -  \log|x| \int_{\mathbb{R}^2_+} |u|^p \longrightarrow 0,
 \qquad \text{as} \; x \longrightarrow \infty .$$
 Here, we view $u \in X_{odd}$ defined on the upper halfspace. We will use this asymptotics in the proof of axial symmetry.
\end{rem}
The following is the general Mountain Pass Lemma for Cerami sequence, see in \cite{LW}.
\begin{lem}\label{lem2.5}
Assume $X$ is Banach space, $M$ is a metric space, $M_0 \subset M $ is a closed subspace, $\Gamma_0 \subset C(M_0;X) $. Define
$$\Gamma:= \big\{ \gamma \in  C(M;X): \gamma \big|_{M_0} \in \Gamma_0      \big\} .$$
If $I \in C^1(X;\mathbb{R})$ satisfies 
$$ \infty > c:= \inf\limits_{\gamma \in  \Gamma} \sup\limits_{t \in M} I(\gamma(t)) >a:=
\inf\limits_{\gamma_0 \in  \Gamma_0} \sup\limits_{t \in M_0} I(\gamma_0(t)),  $$
then for each $ \epsilon \in (0,\frac{c-a}{2}), \delta >0, \gamma \in \Gamma $ with  $\sup\limits_{t \in M} I(\gamma(t)) \leq c+\epsilon ,$ 
there exists a $u \in X$ such that
\begin{enumerate}
	\item  $c-2\epsilon \leq I(u) \leq c+2\epsilon ;$
	\item $ dist(u,\gamma(M)) \leq 2\delta ;$
	\item $ \|I^{\prime}(u)\| _{X^{\prime}} (1+\|u\|_{X})\leq \frac{8\epsilon}{\delta}. $
\end{enumerate}
\end{lem}

%%%%%%%%%%%%%%%%%%%%%%%%%%%%%%%%%%%%%%%%%%%%%%%%%%%%%%%%%%%%%%%%%%%%%%%%%%%%%%%%%%%%%%%%%%%%%%%%%%%%%%%%%%%%%%%%%%
%%%%%%%%%%%%%%%%%%%%%%%%%%%%%%%%%%%%%%%%%%%%%%%%%%%%%%%%%%%%%%%%%%%%%%%%%%%%%%%%%%%%%%%%%%%%%%%%%%%%%%%%%%%%%%%%%%%%%%%%%%%%%%%%%%%%%%
%%%%%%%%%%%%%%%%%%%%%%%%%%%%%%%%%%%%%%%%%%%%%%%%%%%%%%%%%%%%%%%%%%%%%%%%%%%%%%%%%%%%%%%%%%%%%%%%%%%%%%%%%%%%%%%%%%%%%%%%%%%%%%%%%%%%%%%%%%%%%%%%%
\section{proof of theorem $1.1$ }
We use the idea of \cite{CW,DZ}. First we verify the Cerami compactness property of the action functional on the closed subspace $X_{odd}.$
Then using the Mountain Pass Lemma \ref{lem2.5} , we can creat the Cerami sequence. Hence by the compactness, we get a  critical point in $X_{odd}.$ Accordding to the Palais' principle of symmetric criticality, see \cite{P}, it's a critical point in $X.$
\begin{prop}\label{prop3.1}
Let $(u_n)$ be a sequence in $L^p(\mathbb{R}^2)$, s.t. $ u_n \stackrel{a.e.}{\longrightarrow} u \in L^p(\mathbb{R}^2)\setminus \{0 \} . $
 $(v_n)$ be a bounded sequence in $ L^p $ s.t. $ \sup\limits_n B_1(|u_n|^p,|v_n|^p) < \infty .$
Then, there exists $n_0\in \mathbb{Z} $ and $C>0$ s.t. $\|v_n\|_* \leq C $ for $ n\geq n_0 $.
Furthermore, if $B_1(|u_n|^p,|v_n|^p) \to 0  $ and $ \|v_n\|_{L^p} \to 0,  $
then $ \| v_n \|_* \to 0. $	
\end{prop}
\begin{prop}
	Let $ (\widetilde{u_n})$ be a bounded sequence in $X $ such that $ \widetilde {u_n} \rightharpoonup u  $ weakly and a.e. in $X$.
	Then up to a subsequence $ B_1(|\widetilde {u_n}|^p, |\widetilde {u_n}|^{p-2} u (\widetilde {u_n} -u)  )  \longrightarrow 0 .$
\end{prop}
The two propositions can be seen in \cite{CW} for $p=2$ and in \cite{DZ} for $p \geq 2$. 
Based on this two propositions, we can verify the Cerami compactness property for the action functional.
\begin{lem} \label{lem3.3}
Let $(u_n) \subset X_{odd} $ satisfied
\begin{eqnarray} \label{eq3.1}
I(u_n) \longrightarrow d>0, \quad \|I^{\prime}(u_n)\|_{X^{\prime}_{odd}} \big( 1+\|u_n\|_X   \big) \longrightarrow 0, \quad \text{as} \:n \to \infty  . 
\end{eqnarray}
Then up to a subsequence, there exist points $(x_n) \subset \mathbb{Z}^2$, such that
$$ u_n(\cdot - x_n) \longrightarrow u \quad  \text{strongly in} \quad X_{odd}, \quad \text{as} \quad  n\to \infty ,$$
for some nonzero critical point $ u \in X_{odd} $ of $I.$	
\end{lem}
Below we give a refined and rigorous proof for all $p \geq 2$ for this key lemma.
\begin{proof}
 For clarity, We divide the proof into several steps.\\
 \textbf{Step 1}: If $(u_n)$  satisfy condition \eqref{eq3.1}, then $(u_n)$ is bounded in $H^1.$	In fact, we have
$$ o(1)= \big\langle I^{\prime}(u_n),u_n \big \rangle =\|u_n\|_{H^1}^2 + V_0(u_n) , $$
$$ d \longleftarrow I(u_n)=\frac{1}{2}\|u\|_{H^1}^2 + \frac{1}{2p}V_0(u), $$
then $d-\frac{1}{2p}o(1) = \bigg( \frac{1}{2} -\frac{1}{2p} \bigg) \|u_n\|_{H^1}^2, $ so $(u_n)$ is bounded in $H^1.$	\\
\textbf{Step 2}: We claim $(u_n)$ is non-vanishing:
$$ \liminf\limits_{n \to \infty} \sup\limits_{x \in \mathbb{Z}^2} \int_{B_2(x)} u_n^2(y) dy >0   .$$
If not, by the Lion's vanishing lemma, see \cite{S,W}, for each $s>2$, we have $u_n \to 0 $ in $L^s.$ From
$$   o(1)= \big\langle I^{\prime}(u_n),u_n \big \rangle =\|u_n\|_{H^1}^2 + V_1(u_n) -V_2(u_n)   , $$
we get 
$$ \|u_n\|_{H^1}^2 + V_1(u_n) = o(1) + V_2(u_n) .$$
Substituted into $I(u_n)$, yield
$$ d \longleftarrow  I(u_n)=\frac{1}{2}\|u_n\|_{H^1}^2 + \frac{1}{2p}  \bigg( V_1(u_n)  - V_2(u_n)  \bigg)  \longrightarrow 0, $$
absurd. Theorefore there exist points $(x_n) \in \mathbb{Z}^2 $ such that $ \inf_n \int_{B_2(0)} u_n^2(x -x_n)dx >0 .$
Now we define the translation functions $  \widetilde {u_n}=u_n(\cdot - x_n) .$ Also $ \widetilde {u_n} \rightharpoonup u $ in $H^1$
 for some $u \in H^1.$ By the nonvanishing lemma and local compactness, $u \not \equiv 0.$  \\
\textbf{Step 3}: $(\widetilde{u_n})  $ is bounded in $L^p(d \mu).$ In fact, from
\begin{align*}
V_1(\widetilde {u_n}) &= \big\langle I^{\prime}(\widetilde {u_n}) ,\widetilde {u_n}\big \rangle - \|\widetilde {u_n}\|_{H^1}^2 +
V_2(\widetilde {u_n})  \\
&=\big\langle I^{\prime}(u_n) ,u_n \big \rangle - \|u_n\|_{H^1}^2 +
V_2(u_n)  \\
&=o(1) -  \|u_n\|_{H^1}^2 + V_2(u_n),
\end{align*}
we get $ V_1(\widetilde {u_n}) $ is bounded. By the Proposition \ref{prop3.1}, we get   $ \|\widetilde {u_n} \|_*^p $ is bounded.
Hence $(\widetilde{u_n})  $ is bounded in $X$, which is compactely embedding in $L^s$ for all $s \geq p.$\\
\textbf{Step 4}:   $ \widetilde {u_n} \longrightarrow u $ strongly in $X_{odd}$.
First, we claim that:
$$ \big\langle I^{\prime}(\widetilde {u_n}) ,\widetilde {u_n} - u  \big \rangle \longrightarrow 0 \; \text{as} \; n \to \infty. $$
In fact, by the $ \mathbb{Z}^2- $translation invariance, we have
$$ \bigg| \big\langle I^{\prime}(\widetilde {u_n}) ,\widetilde {u_n} - u  \big \rangle \bigg|
= \bigg| \big\langle I^{\prime}(u_n) ,u_n - u(\cdot + x_n)  \big \rangle \bigg|  \leq
\|I^{\prime}(u_n)\|_{X^{\prime}_{odd}} \bigg( \|u_n\|_X + \|u(\cdot+x_n)\| _X \bigg)  $$
Now, we estimate the last two terms in the following way:
\begin{align*}
\bigg| \|u_n\|_*^p - \log(1+|x_n|) \|\widetilde {u_n}\|_{L^p}^p  \bigg|  &=
\bigg| \int \log(1+|x-x_n|) |\widetilde {u_n}|^p dx-  \int \log(1+|x_n|) |\widetilde {u_n}|^p dx     \bigg| \\
& = \bigg|  \int_A \log \bigg( \frac{1+|x-x_n|}{1+|x_n|}\bigg) |\widetilde {u_n}|^p  
+\int_B \log \bigg( \frac{1+|x-x_n|}{1+|x_n|}\bigg) |\widetilde {u_n}|^p  \bigg|\\
&= \bigg| \int_A \log \bigg( \frac{1+|x-x_n|}{1+|x_n|}\bigg) |\widetilde {u_n}|^p
-  \int_B \log \bigg( \frac{1+|x|}{1+|x-x_n|}\bigg) |\widetilde {u_n}|^p  \bigg|, \\
\end{align*}
where $ A=\{ |x-x_n| \geq |x_n| \}, B=\{ |x-x_n| \leq |x_n| \}.  $  Then we choose $ \delta \in (0,1) $ fixed, 
set $ D_1=\{ |x-x_n| \leq \delta \} \cap B, D_2 = \{ \delta  \leq |x-x_n| \leq |x_n|  \}  ,$ and each term is 
bounded by a constant independed of $n$:
\begin{align*}
 \int\limits_{|x-x_n| \geq |x_n| } \log \bigg( \frac{1+|x-x_n|}{1+|x_n|}\bigg) |\widetilde {u_n}|^p  & \leq
  \int\limits_{A } \log \bigg( \frac{(1+|x|)(1+|x_n|)}{1+|x_n|}\bigg) |\widetilde {u_n}|^p \\
  & \leq \int\limits_{A } \log \big( 1+|x| \big) |\widetilde {u_n}|^p \\
  & \leq C ,
\end{align*}
%%%%%%%%%%%%%%%%%%%%%%%%%%%%%%%%%%%%%%%%%%%%%%%%%%%%%%%%%%%%%%%%%%%%%%%%%%%%%%%%%%%%%%%%%%%%%%%%%%%%%%%%%%%%%%%%%%%%%%%%%%%%%%%%%
\begin{align*}
\int\limits_{\delta  \leq |x-x_n| \leq |x_n|} \log \bigg( \frac{1+|x_n|}{1+|x-x_n|}\bigg) |\widetilde {u_n}|^p
&\leq  \int\limits_{D_2} \log \bigg( 1+ \frac{|x|}{1+|x-x_n|}\bigg) |\widetilde {u_n}|^p \\
& \leq  \int\limits_{D_2} \log \bigg( 1+ \frac{|x|}{1+\delta}\bigg) |\widetilde {u_n}|^p\\
& \leq \int\limits_{D_2} \log \big( 1+ |x| \big) |\widetilde {u_n}|^p \\
& \leq C ,
\end{align*}
%%%%%%%%%%%%%%%%%%%%%%%%%%%%%%%%%%%%%%%%%%%%%%%%%%%%
\begin{align*}
\int\limits_{\{ |x-x_n| \leq \delta \} \cap B } \log(1+|x_n|) |\widetilde {u_n}|^p
& \leq  \int\limits_{D_1}  \log(1+|x|+ \delta ) |\widetilde {u_n}|^p \\
&\leq   \int\limits_{D_1}  \log \big( 2 (1+|x|)^2 \big) |\widetilde {u_n}|^p \\
&\leq 2 \int\limits_{D_1} |\widetilde {u_n}|^p + 2  \int\limits_{D_1} \log(1+|x|) |\widetilde {u_n}|^p \\
& \leq C ,
\end{align*}
%%%%%%%%%%%%%%%%%%%%%%%%%%%%%%%%%%%%%%%%%%%%%%%%%%%%%%%
\begin{align*}
\int\limits_{\{ |x-x_n| \leq \delta \} \cap B } \log(1+|x-x_n|) |\widetilde {u_n}|^p
&\leq \bigg( \int\limits_{D_1} \big( \log(1+|x-x_n|) \big) ^2   \bigg)^{\frac{1}{2}}  \cdot
\bigg( \int\limits_{D_1}  |\widetilde {u_n}|^{2p} \bigg)^{\frac{1}{2}} \\
&=  \bigg( \int\limits_{|y| \leq \delta} \big( \log(1+|y|) \big) ^2   \bigg)^{\frac{1}{2}} \cdot
\bigg( \int\limits_{D_1}  |\widetilde {u_n}|^{2p} \bigg)^{\frac{1}{2}}  \\
&\leq C_{\delta} .
\end{align*}
Hence the above estimates yield
$$ \bigg| \|u_n\|_*^p - \log(1+|x_n|) \|\widetilde {u_n}\|_{L^p}^p  \bigg| \leq C, $$
 since $u \not\equiv 0$, we get 
$$  \bigg| \|u_n\|_*^p -  C_2 \log(1+|x_n|) \bigg|  \leq C_1  .  $$
Also for $ \| u(\cdot +x_n)\|_*^p $, we have
\begin{align*}
\|u(\cdot + x_n)\|_*^p &= \int \log(1+|x|) |u|^p(x+x_n)dx\\
&=\int \log(1+|x-x_n|) |u|^p \\
&\leq \int \log(1+|x|) |u|^p + \int \log(1+|x_n|) |u|^p\\
&\leq C_3 +C_4 \log(1+|x_n|)
\end{align*}
Combining the two estimates, we have
$$ \bigg( \|u_n\|_X + \|u(\cdot + x_n)\| _X \bigg) \leq C \bigg( 1+ \|u_n\|_X  \bigg) .$$
Then by the assumption (\ref{eq3.1}) in Lemma $3.3$, we have
$$ \big\langle I^{\prime}(\widetilde {u_n}) ,\widetilde {u_n} - u  \big \rangle \longrightarrow 0  $$
as claimed.
But on the other side, we get
\begin{align*}
o(1)&=\big\langle I^{\prime}(\widetilde {u_n}) ,\widetilde {u_n} - u  \big \rangle\\
&= \big\langle I^{\prime}(\widetilde {u_n}) ,\widetilde {u_n} \big \rangle -\big\langle I^{\prime}(\widetilde {u_n}) , u  \big \rangle\\
&= \|\widetilde {u_n}\|_{H^1}^2 -\|u\|_{H^1}^2 +o(1)+ \bigg\langle V_0^{\prime}(\widetilde {u_n}),\widetilde {u_n}-u \bigg\rangle.
\end{align*}
Estimating the each term yield
$$ \bigg\langle V_2^{\prime}(\widetilde {u_n}),\widetilde {u_n}-u \bigg\rangle \longrightarrow 0 $$
by the compact embedding of $ X \hookrightarrow  \hookrightarrow L^s $ and the Hardy-Littlewood-Sobolev Inequality.
Now we estimate the $V_{1}^{\prime}:$ 
\begin{align*}
\bigg\langle V_1^{\prime}(\widetilde {u_n}),\widetilde {u_n}-u \bigg\rangle
&= B_1 \bigg( |\widetilde {u_n}|^{p},|\widetilde {u_n}|^{p-2}  (\widetilde {u_n} (\widetilde {u_n}-u)     \bigg) \\
&=B_1 \bigg( |\widetilde {u_n}|^{p},|\widetilde {u_n}|^{p-2} \big(  (\widetilde {u_n}-u)^2 + u (\widetilde {u_n}-u) \big)    \bigg)\\
&= B_1 \bigg( |\widetilde {u_n}|^{p}, |\widetilde {u_n}|^{p-2} |\widetilde {u_n}-u|^2       \bigg)
+B_1 \bigg( |\widetilde {u_n}|^{p}, |\widetilde {u_n}|^{p-2}   u (\widetilde {u_n}-u)    \bigg).
\end{align*}
But $  B_1 \bigg( |\widetilde {u_n}|^{p}, |\widetilde {u_n}|^{p-2}   u (\widetilde {u_n}-u)     \bigg) \to 0 $ by the Proposition $ 3.2 $.
Let $ v_n^p = |\widetilde {u_n}|^{p-2} |\widetilde {u_n}-u|^2  $, then
$$ B_1 \bigg( |\widetilde {u_n}|^{p}, |\widetilde {u_n}|^{p-2} |\widetilde {u_n}-u|^2   \bigg)
= B_1 \bigg( |\widetilde {u_n}|^{p}, |v_n|^{p}       \bigg)  \geq 0 $$
and we get
\begin{align*}
o(1) &=\big\langle I^{\prime}(\widetilde {u_n}) ,\widetilde {u_n} - u  \big \rangle\\
&=o(1)+\|\widetilde {u_n}\|_{H^1}^2 -\|u\|_{H^1}^2  + B_1 \big( |\widetilde {u_n}|^{p}, |v_n|^{p} \big)\\
&\geq o(1)+\|\widetilde {u_n}\|_{H^1}^2 -\|u\|_{H^1}^2,
\end{align*}
which implies $ \|\widetilde {u_n}\|_{H^1}^2  \to \|u\|_{H^1}^2 $ and $  B_1 \big( |\widetilde {u_n}|^{p}, |v_n|^{p} \big) \to 0 .$
So we get the strong convergence in $ H^1:\; \widetilde {u_n} \to u  $. Again, by the compact embedding of $ X \hookrightarrow \hookrightarrow L^p $,
we get $ v_n^p = |\widetilde {u_n}|^{p-2} |\widetilde {u_n}-u|^2 \to 0 $ in $ L^1 $. By the Proposition $ 3.1 $,
we get $ \|v_n\|_*^p \to 0 $.\\
Now, applying the Lemma $ 2.1 $, we have
\begin{align*}
o(1) &=\int \log(1+|x|) |\widetilde {u_n}|^{p-2} |\widetilde {u_n}-u|^2 \\
&=\int \bigg(  |\widetilde {u_n}-u|^{p-2}  + |\widetilde {u_n}|^{p-2} -  |\widetilde {u_n}-u|^{p-2}  \bigg)  |\widetilde {u_n}-u|^2 d\mu \\
& =\int  |\widetilde {u_n}-u|^{p} d \mu  +
\int  \bigg( |\widetilde {u_n}|^{p-2} -  |\widetilde {u_n}-u|^{p-2}  \bigg)  |\widetilde {u_n}-u|^2 d\mu \\
&\geq \int  |\widetilde {u_n}-u|^{p} d \mu - \epsilon \int  |\widetilde {u_n}|^{p-2}  |\widetilde {u_n}-u|^2 d\mu
- C_{\epsilon} \int |u|^{p-2} |\widetilde {u_n}-u|^2 d\mu\\
& \geq  \int  |\widetilde {u_n}-u|^{p} d \mu - M \epsilon - C_{\epsilon} o(1),
\end{align*}
So we get $ \| \widetilde {u_n}-u \|_* \to 0 $. Combining with the $ H^1 $ convergence, we get the
strong convergence in $ X: \; \| \widetilde {u_n}-u \|_X  \to 0 $.\\
 \textbf{Step 5}: We prove $ u $ is the critical point: $ I^{\prime}(u)=0 $. This is easily checked. \\
Let $v \in X_{odd} $, as we have already shown
$$ \|v(\cdot + x_n)\|_* \leq C(1+ \log(1+|x_n|)) \leq C(1+\|u_n\|_*)  . $$
By this and $ u \neq 0 $, we have $ \|v (\cdot + x_n)\|_X \leq C(1+\|u_n\|_X) $.
Notice that $ \big \langle I^{\prime}(u),v  \big \rangle  =\lim \limits_{n \to \infty}  \big\langle I^{\prime}(\widetilde {u_n}) , v  \big \rangle .$
But we also have
\begin{align*}
\bigg|  \big\langle I^{\prime}(\widetilde {u_n}) , v  \big \rangle    \bigg| &=
\big|  \langle I^{\prime}(u_n) ,  v(\cdot + x_n)    \rangle \big| \\
&\leq  \|I^{\prime}(u_n)\|_{X^{\prime}_{odd}}  \| v(\cdot + x_n) \|_X\\
& \leq   C \|I^{\prime}(u_n)\|_{X^{\prime}_{odd}}  \bigg( 1+\|u_n\|_X  \bigg) \longrightarrow 0
\end{align*}
by the assumption. Hence $ \big \langle I^{\prime}(u),v  \big \rangle =0 $. And we finish the proof.
\end{proof}
%%%%%%%%%%%%%%%%%%%%%%%%%%%%%%%%%%%%%%%%%%%%%%%%%%%%%%%%%%%%%%%%%%%%%%%%%%%%%%%%%%%%%%%%%%%%%%%%%%
\begin{proof}[Proof of Theorem $1.1$] (1): $c_{mp,odd} \geq c_{mp} >0$ obviously.\\
	(2): First we use the Mountain Pass Lemma \ref{lem2.5} to construct the Cerami sequence $(u_n) \subset X_{odd}$, then applying the Cerami compactness property of Lemma \ref{lem3.3} for the action functional , we can extract a subsequence converges to a 
	nonzero critical point $u$, and $I(u) =c_{mp,odd}$. By the Palais' principle of symmetric criticality, 
	$u$ is a critical point in $X$, and satisfies the corresponding properties of Lemma $2.3$.\\
	(3): First we notice $ c_{g,odd}\geq  c_{odd} = c_{mm,odd} \geq c_{mp,odd} ,$ but also $ c_{g,odd} \leq c_{mp,odd} .$ 
	so $ c_{g,odd}= c_{odd}= c_{mm,odd}= c_{mp,odd} .$ The  equality $  c_{odd}= c_{mm,odd} $ is by the monotonicity of $I(tu)$ for $t$,
	 see \cite{zbMATH06524159,DZ}.\\
	(4):This is obvious, since all the ground state solutions have constant sign, see \cite{zbMATH06524159,DZ}, and it can not be zero on the $\{x=(x_1,x_2): x_2 =0\}. $
\end{proof}
%%%%%%%%%%%%%%%%%%%%%%%%%%%%%%%%%%%%%%%%%%%%%%%%%%%%%%%%%%%%%%%%%%%%%%%%%%%%%%%%%%%%%%%%%%%%%%%%%%%%%%%%%%%%%%%%%%%%%%%%
%%%%%%%%%%%%%%%%%%%%%%%%%%%%%%%%%%%%%%%%%%%%%%%%%%%%%%%%%%%%%%%%%%%%%%%%%%%%%%%%%%%%%%%%%%%%%%%%%%%%%%%%%%%%%%%%%%%%%%%%
%%%%%%%%%%%%%%%%%%%%%%%%%%%%%%%%%%%%%%%%%%%%%%%%%%%%%%%%%%%%%%%%%%%%%%%%%%%%%%%%%%%%%%%%%%%%%%%%%%%%%%%%%%%%%%%%%%%%%%%%
%%%%%%%%%%%%%%%%%%%%%%%%%%%%%%%%%%%%%%%%%%%%%%%%%%%%%%%%%%%%%%%%%%%%%%%%%%%%%%%%%%%%%%%%%%%%%%%%%%%%%%%%%%%%%%%%%%%%%%%%
%%%%%%%%%%%%%%%%%%%%%%%%%%%%%%%%%%%%%%%%%%%%%%%%%%%%%%%%%%%%%%%%%%%%%%%%%%%%%%%%%%%%%%%%%%%%%%%%%%%%%%%%%%%%%%%%%%%%%%%%
\section{Proof of Theorem $1.2$}
In this section, we  prove all the minimal action odd solution are axially symmetric and $\frac{\partial u}{\partial x_1} <0$
along the ray starting from the axis in $x_1$ direction. To prove this, we first reformulate the minimal action odd 
solution problem into a ground state problem in the upper half plane for some similar but more complex equation. 
Then we will carefully apply the method of moving plane to this equation to derive the axial symmetry.
\begin{prop}
For every $v \in X_{odd}(\mathbb{R}^2),$ we have 
$$ I(v)= \widetilde{I}(v \big|_{\mathbb{R}^2_+} ) , $$
where the new action functional 
$ \widetilde{I}: \widetilde{X}:=H^1_0( \mathbb{R}^2_+) \cap L^p(d \mu)\big|_{ \mathbb{R}^2_+} \to \mathbb{R} $ 
 is defined on the upper halfplane by
 \begin{eqnarray*}
 \widetilde{I}(v)&= \int_{\mathbb{R}^2_+} (|\nabla v|^2 + |v|^2)dx + \frac{1}{2p} 
 \bigg(  2 \int_{\mathbb{R}^2_+} \int_{\mathbb{R}^2_+} \big( \log |x-y|\big) |v|^p (x)|v|^p(y)dxdy\\
 & +  \int_{\mathbb{R}^2_+} \int_{\mathbb{R}^2_+} \big[ \log \big((x_1 -y_1)^2 + (x_2+y_2)^2 \big) \big]  |v|^p (x)|v|^p(y)dxdy \bigg) .
 \end{eqnarray*}
In pariticular, if $u$ is the minimal action odd solution, then $u \in  \widetilde{\mathcal{N}}$ with 
$$ \widetilde{\mathcal{N}}:= \bigg\{  w \in \widetilde{X} : w \not\equiv 0, \langle \widetilde{I}^{\prime}(w),w \rangle =0 \bigg\} ,$$
and we have
$$ \widetilde{I}(u)= \inf_{\widetilde{\mathcal{N}}} \widetilde{I}(w) . $$
\end{prop}
\begin{proof}
  The proofs are direct computations  and use the fact $u \in X_{odd}$ if and only if $u \big|_{\mathbb{R}^2_+}  \in \widetilde{X}. $
\end{proof}
From now on, we will freely view $u \in X_{odd}$ or $u \in \widetilde{X}(\mathbb{R}^2_+).$
Note that the existence of minimum problem for $  \widetilde{I}(u)= \inf_{\widetilde{\mathcal{N}}} \widetilde{I}(w)  $
have already been proved by the Theorem $1$. Also the minimal odd solution $u$ satisfies the new Euler-Lagrange equation:
\begin{eqnarray}\label{eq4.1}
-\Delta u +u + \frac{1}{2}H |u|^{p-2}u =0 \qquad \text{in} \; \mathbb{R}^2_+,
\end{eqnarray}
where $H=H_1+H_2$ is defined by
\begin{align*}
H_1(x)&:= 2 \int_{\mathbb{R}^2_+}  \big( \log |x-y|\big) |u|^p(y)dy, \\
H_2(x) &:= \int_{\mathbb{R}^2_+} \big[ \log \big((x_1 -y_1)^2 + (x_2+y_2)^2 \big) \big]  |u|^p(y)dy .
\end{align*}
Recall that we have shown in the Remark $2.4$ that 
$$ \frac{1}{2}H_2(x) - \log|x| \int_{\mathbb{R}^2_+} |u|^p \longrightarrow 0, \qquad \text{as} \; x \longrightarrow \infty.  $$
From this, we can see $H_2$ is at most $\log$ growth, as $H_1$ does. But we need to  give an explicit
boundedness in terms of $\|u\|_{\widetilde{X}}$ to show $H_2$ is  well-defined on the $\widetilde{X}$.
%%%%%%%%%%%%%%%%%%%%%%%%%%%%%%%%%%%%%%%%%%%%%%%%%%%%%%%%%%%%%%%%%%%%%%%%%%%%%%%%%%%%%%%%%%%%%%%%%%%%%%%%%%%%%%%%%%%%%%%%%%%%%%%%%%%%%%%
%%%%%%%%%%%%%%%%%%%%%%%%%%%%%%%%%%%%%%%%%%%%%%%%%%%%%%%%%%%%%%%%%%%%%%%%%%%%%%%%%%%%%%%%%%%%%%%%%%%%%%%%%%%%%%%%%%%%%%%%%%%%%%%%%%%%%%%%%%
%%%%%%%%%%%%%%%%%%%%%%%%%%%%%%%%%%%%%%%%%%%%%%%%%%%%%%%%%%%%%%%%%%%%%%%%%%%%%%%%%%%%%%%%%%%%%%%%%%%%%%%%%%%%%%%%%%%%%%%%%%%%%%%%%%%%%%%%%%%%
\begin{prop}
$H_2(x) = F(x)-G(x),$  where $ F,G$ are nonnegative functions bounded by  
\begin{align*}
 F(x) &\leq C \bigg( \int_{\mathbb{R}^2_+} |u|^p dy +  \int_{\mathbb{R}^2_+} |u|^p d{\mu}  +\log(1+|x|) \int_{\mathbb{R}^2_+} |u|^p d\mu  \bigg), \\
 G(x)  &\leq C\int_{\mathbb{R}^2_+} \frac{1}{|x-y|} |u|^p(y)dy.
 \end{align*}
\end{prop}
\begin{proof}
This is a regular computation. First we define spherical cap over the upper half plane:
$$\Omega_1(x) := \big\{y=(y_1,y_2) \in \mathbb{R}^2_+ : |x_1-y_1|^2 +|x_2 +y_2|^2 \leq 1 \big\},\quad \Omega^{c}_1(x):= \mathbb{R}^2_+ \setminus \Omega_1(x).  $$
Then we have
\begin{align*}
H_2(x) &:=   \bigg( \int_{\Omega^{c}_1(x) } 
          + \int_{\Omega_1(x) }  \bigg)  \big[ \log \big((x_1 -y_1)^2 + (x_2+y_2)^2 \big) \big]  |u|^p(y)dy \\
       &= F(x)-G(x).
\end{align*}
We estimate $F(x)$ and $G(x)$ in the following way:
\begin{align*}
F(x) & =  \int_{\Omega^{c}_1(x) }  \big[ \log \big((x_1 -y_1)^2 + (x_2+y_2)^2 \big) \big]  |u|^p(y)dy \\
   & \leq  \int_{\Omega^{c}_1(x) }   \big[ \log \big( (1+ (x_1 -y_1)^2)(1+ (x_2+y_2)^2 )\big) \big]  |u|^p(y)dy\\
   & \leq  \int_{\Omega^{c}_1(x) }   \big[ \log \big( 1+ |x_1 -y_1|^2\big) \big] |u|^p(y)dy   
        +  \int_{\Omega^{c}_1(x) }   \big[ \log \big( 1+ |x_2+y_2|^2\big) \big] |u|^p(y)dy  \\
   & \leq  \int_{\Omega^{c}_1(x) }  \big[ \log  2 (1+ |x_1 -y_1|)^2  \big] |u|^p(y)dy 
        +  \int_{\Omega^{c}_1(x) }  \big[ \log  2 (1+ |x_2 +y_2|)^2  \big] |u|^p(y)dy  \\
   &\leq C  \int_{\Omega^{c}_1(x) }  |u|^p(y)dy + C   \int_{\Omega^{c}_1(x) }   \big[ \log  \big( (1+ |x -y|) \big) \big] |u|^p(y)dy \\
    &\ \  \  +C \int_{\Omega^{c}_1(x) }    \big[ \log  \big( (1+ |x|)(1+ |y|) \big) \big] |u|^p(y)dy  \\
         & \leq  C \bigg( \int_{\mathbb{R}^2_+} |u|^p dy +  \int_{\mathbb{R}^2_+} |u|^p d{\mu}  +\log(1+|x|) \int_{\mathbb{R}^2_+} |u|^p d\mu  \bigg).
\end{align*}
For $G(x)$ we have
\begin{align*}
G(x) &=  \int_{\Omega_1(x) }  \bigg( \log   \frac{1}{  (x_1 -y_1)^2 + (x_2+y_2)^2   }   \bigg) |u|^p(y)dy \\
     & \leq  \int_{\Omega_1(x) }  \bigg( \log   \frac{1}{  (x_1 -y_1)^2 + (x_2-y_2)^2   }   \bigg) |u|^p(y)dy \\
     & \leq  C\int_{\mathbb{R}^2_+} \frac{1}{|x-y|} |u|^p(y)dy.
\end{align*}
This is done.
	\end{proof}
%%%%%%%%%%%%%%%%%%%%%%%%%%%%%%%%%%%%%%%%%%%%%%%%%%%%%%%%%%%%%%%%%%%%%%%%%%%%%%%%%%%%%%%%%%%%%%%%%%%%%%%%%%%%%%%%%%%%%%%%%%%%%%%%%%%%%
Now we can prove
\begin{prop}
	If $u$ is the minimal odd solution, then $u >0$ or $u<0$ in $\mathbb{R}^2_+.$
\end{prop}
\begin{proof}
By the characterzation of $  \widetilde{I}(u)= \inf_{\widetilde{\mathcal{N}}} \widetilde{I}(w) $ for $u \not \equiv 0$, we see $|u|$ is also the minimum for the new action
functional $\widetilde{I}$, and satisfies the new Euler-Lagrange equation \eqref{eq4.1}. Applying the maximum principle of Serrin, $|u| >0.$ So $u$ has constant sign in the upper 
halfplane $ \mathbb{R}^2_+.$
	\end{proof}
\par
Based on the semilinear elliptic equation \eqref{eq4.1} and the positivity of $u$ in $\mathbb{R}^2_+$,  we will use the method of moving plane carefully to deduce the symmetry property of the solutions. First we fix some solution $u$ and all the constants  below will depend on this solution $u$, but independent of the moving plane $T_{\lambda}.$
To carry out the method of moving plane, we define $T_{\lambda} := \big\{ x=(x_1,x_2) \in \mathbb{R}^2_+ : x_1= \lambda \big\} $, 
where we will move $ T_{\lambda} $ from $ \lambda = - \infty$ to some limiting position. Define $ \Sigma_{\lambda} $  is the left part of $T_{\lambda}:$
$ \Sigma_{\lambda}:=  \big\{ x=(x_1,x_2) \in \mathbb{R}^2_+ : x_1< \lambda \big\} .$ Let  $x^{\lambda}$ be the reflection point with respect to $T_{\lambda}$ for the point
 $x \in \Sigma_{\lambda}$: $ x^{\lambda} =(2\lambda-x_1,x_2) =(x_1^{\lambda},x_2^{\lambda})$. We will compare the values of $u$ at 
 the points $x^{\lambda}$ and $ x.$ For this, we let $u_{\lambda}(x) :=u (x^{\lambda}),$ $ w_{\lambda}:= u_{\lambda} -u ,$
and $ L_{\lambda}:=H_1(x^{\lambda}) -H_1(x) ,$  $ M_{\lambda}(x):= H_2(x^{\lambda})  - H_2(x).$ We first need the integral representation of $L_{\lambda}$ and $M_{\lambda}.$
\begin{prop}
For $ x \in \Sigma_{\lambda},$ we have
\begin{align*}
L_{\lambda}(x) &=  2 \int_{\Sigma_{\lambda}}  \bigg( \log \frac{|x-y|}{|x-y^{\lambda}|} \bigg) \bigg( |u_{\lambda}|^p -|u|^p   \bigg)  dy ,\\
M_{\lambda}(x) &=  \int_{\Sigma_{\lambda}}     \bigg( \log \frac{|x_1-y_1|^2+|x_2+y_2|^2}{|x_1-y^{\lambda}_1|^2+ |x_2+y_2|^2 } \bigg) \bigg( |u_{\lambda}|^p -|u|^p   \bigg)  dy.
\end{align*}
\end{prop}
\begin{proof}
We check it directly. 
\begin{align*}
H_1(x)&= 2 \int_{\mathbb{R}^2_+}  \big( \log |x-y|\big) |u|^p(y)dy \\
      &= 2 \int_{\Sigma_{\lambda}}  \big( \log |x-y|\big) |u|^p(y)dy  + 2  \int_{ \mathbb{R}^2_+ \setminus  \Sigma_{\lambda}}  \big( \log |x-y|\big) |u|^p(y)dy \\
      &= 2 \int_{\Sigma_{\lambda}}  \big( \log |x-y|\big) |u|^p(y)dy + 2 \int_{\Sigma_{\lambda}}  \big( \log |x-y^{\lambda}|\big) |u_{\lambda}|^p(y)dy \\
H_2(x) &:= \int_{\mathbb{R}^2_+} \big[ \log \big((x_1 -y_1)^2 + (x_2+y_2)^2 \big) \big]  |u|^p(y)dy \bigg);\\
       &= \int_{\Sigma_{\lambda}} \big[ \log \big((x_1 -y_1)^2 + (x_2+y_2)^2 \big) \big]  |u|^p(y)dy  \\
       & \ \ \ +  
          \int_{\mathbb{R}^2_+ \setminus \Sigma_{\lambda}} \big[ \log \big((x_1 -y_1)^2 + (x_2+y_2)^2 \big) \big]  |u|^p(y)dy \\
       &=  \int_{\Sigma_{\lambda}} \big[ \log \big((x_1 -y_1)^2 + (x_2+y_2)^2 \big) \big]  |u|^p(y)dy  \\
       & \ \ \ +
           \int_{\Sigma_{\lambda}} \big[ \log \big((x_1 -y_1^{\lambda})^2 + (x_2+y_2)^2 \big) \big]  |u_{\lambda}|^p(y)dy .
\end{align*}
	Substituted into  $ L_{\lambda}:=H_1(x^{\lambda}) -H_1(x) ,$  $ M_{\lambda}(x):= H_2(x^{\lambda})  - H_2(x),$ yield the integral representations. 
\end{proof}
\begin{proof}[Proof of Theorem $1.2$]
By the Euler-Lagrange equation \eqref{eq4.1}, we know $  w_{\lambda}= u_{\lambda} -u $ satisfies the equation
\begin{eqnarray}\label{eq4.2}
 -\Delta w_{\lambda} + w_{\lambda} + \frac{1}{2} L_{\lambda} |u_{\lambda}|^{p-2} u_{\lambda} +
 \frac{1}{2} M_{\lambda} |u_{\lambda}|^{p-2} u_{\lambda} + \frac{1}{2}(p-1) H(x) |\psi_{\lambda}|^{p-2} w_{\lambda} =0 ,
 \end{eqnarray}	
where $ \psi_{\lambda}  $ between $u_{\lambda}$ and $ u.$ We define the negative part  $\Sigma_{\lambda}^-$ of $w_{\lambda}$ in $\Sigma_{\lambda}$ by
$$ \Sigma_{\lambda}^- := \bigg \{  x \in \Sigma_{\lambda}: w_{\lambda}(x) =u_{\lambda}(x)-u(x) <0 \bigg\}. $$
Our aim is to show this set is empty $ \Sigma_{\lambda} = \emptyset $ to give $w_{\lambda} =u_{\lambda} -u \geq 0 $, until $T_{\lambda}$ arrive at some limiting position $\lambda_0$, 
which we will have $  w_{\lambda_0}(x) =u_{\lambda_0}(x)-u(x) =0  $, the desired symmetry property. We divide this process of moving 
plane into two steps.\\
\textbf{Step 1:} Start moving the plane from $ \lambda =-\infty.$   \\
We multiply the equation \eqref{eq4.2} by $w_{\lambda} $, and integrate over $ \Sigma_{\lambda}^-.$
Notice that on the set $\Sigma_{\lambda}^-,$ $ 0< u_{\lambda} \leq \psi_{\lambda} \leq u .$ We have
\begin{align*}
\int_{\Sigma_{\lambda}^-} |\nabla w^-_{\lambda} |^2  + \int_{\Sigma_{\lambda}^-} |w^-_{\lambda}| ^2 &= 
- \frac{1}{2} \int_{\Sigma_{\lambda}^-} L_{\lambda} |u_{\lambda}|^{p-2} u_{\lambda} w_{\lambda} 
-\frac{1}{2} \int_{\Sigma_{\lambda}^-} M_{\lambda} |u_{\lambda}|^{p-2} u_{\lambda} w_{\lambda} \\
   &\ \ \ \  - \int_{\Sigma_{\lambda}^-}  \frac{1}{2}(p-1) H(x) |\psi_{\lambda}|^{p-2} (w_{\lambda})^2.
\end{align*}
Since $H(x) \longrightarrow  + \infty,$ we choose $\lambda$ negative enough, such that 
$ \int_{\Sigma_{\lambda}^-} H(x) |\psi_{\lambda}|^{p-2} (w^-_{\lambda})^2  \geq 0 . $
Then we get
\begin{align}\label{eq4.4}
\int_{\Sigma_{\lambda}^-} |\nabla w^-_{\lambda} |^2  + \int_{\Sigma_{\lambda}^-} |w^-_{\lambda}| ^2  
&\leq  \frac{1}{2} \int_{\Sigma_{\lambda}^-}  L_{\lambda}    |u_{\lambda}|^{p-2} u_{\lambda} w^-_{\lambda} +
       \frac{1}{2}  \int_{\Sigma_{\lambda}^-}  M_{\lambda}    |u_{\lambda}|^{p-2} u_{\lambda} w^-_{\lambda}\nonumber \\ 
 & \leq \frac{1}{2} \int_{\Sigma_{\lambda}^-}  L_{\lambda}^+    |u_{\lambda}|^{p-2} u_{\lambda} w^-_{\lambda} +
 \frac{1}{2}  \int_{\Sigma_{\lambda}^-}  M_{\lambda}^+    |u_{\lambda}|^{p-2} u_{\lambda} w^-_{\lambda}.
\end{align}
Now we estimate $L_{\lambda}^+$ and $M_{\lambda}^+$ separately using the integral representations of them. For $x \in \Sigma_{\lambda}^-,$ we have
\begin{align*}
L_{\lambda}(x)  &=  2\int_{\Sigma_{\lambda}}  \bigg( \log \frac{|x-y|}{|x-y^{\lambda}|} \bigg) \bigg( |u_{\lambda}|^p -|u|^p   \bigg)  dy   \\
    &=  2\int_{\Sigma_{\lambda}^+}  \bigg( \log \frac{|x-y|}{|x-y^{\lambda}|} \bigg) \bigg( |u_{\lambda}|^p -|u|^p   \bigg)   
    + 2 \int_{\Sigma_{\lambda}^-}  \bigg( \log \frac{|x-y|}{|x-y^{\lambda}|} \bigg) \bigg( |u_{\lambda}|^p -|u|^p   \bigg) ;    \\
L^+_{\lambda}(x)    &\leq  2 \int_{\Sigma_{\lambda}^-}  \bigg( \log \frac{|x-y|}{|x-y^{\lambda}|} \bigg) \bigg( |u_{\lambda}|^p -|u|^p   \bigg) dy\\
    & =  2\int_{\Sigma_{\lambda}^-}  \bigg( \log \frac{|x-y^{\lambda}|}{|x-y|} \bigg) \bigg( |u|^p - |u_{\lambda}|^p   \bigg) dy \\
    & \leq 2 \int_{\Sigma_{\lambda}^-}  \bigg( \log \big( 1+ \frac{|y^{\lambda}-y|}{|x-y|} \big) \bigg)   \bigg( |u|^p - |u_{\lambda}|^p   \bigg) dy \\
    &\leq C \int_{\Sigma_{\lambda}^-}  \bigg( \log \big( 1+ \frac{2|\lambda-y_1|}{|x-y|} \big) \bigg)    |\phi_{\lambda}|^{p-2} \phi_{\lambda} w^-_{\lambda} dy \\
    &\leq C \int_{\Sigma_{\lambda}^-}  \frac{1}{|x-y|} (\lambda-y_1) |\phi_{\lambda}|^{p-1} w^-_{\lambda} dy\\
    & \leq C  \int_{\Sigma_{\lambda}^-}  \frac{1}{|x-y|} (\lambda-y_1)  |u|^{p-1} w^-_{\lambda} dy.
\end{align*}
So by the Hardy-Littlewood-Sobolev inequality, we have
\begin{align*}
\big\|  L_{\lambda}^+ \big \|_{L^4(\Sigma_{\lambda}^-)} &\leq C \big\| (\lambda-y_1) |u|^{p-1} w^-_{\lambda}     \big\|_{L^{\frac{4}{3}}(\Sigma_{\lambda}^-)}\\
& \leq C \big\| (\lambda-y_1) |u|^{p-1}\big\|_{L^4(\Sigma_{\lambda}^-)}  \big\|  w^-_{\lambda} \big\|_{L^2(\Sigma_{\lambda}^-)}.
\end{align*}
Hence 
\begin{align*}
\int_{\Sigma_{\lambda}^-}  L_{\lambda}^+    |u_{\lambda}|^{p-2} u_{\lambda} w^-_{\lambda}
 & \leq C \big\|L_{\lambda}^+ \big \|_{L^4(\Sigma_{\lambda}^-)}   \big\|u_{\lambda}^{p-1} \big\|_{L^4(\Sigma_{\lambda}^-)} 
\big\|w^-_{\lambda} \big\|_{L^2(\Sigma_{\lambda}^-)} \\
& \leq C \big\| (\lambda-y_1) |u|^{p-1}\big\|_{L^4(\Sigma_{\lambda}^-)} \big\|u^{p-1} \big\|_{L^4(\Sigma_{\lambda}^-)} 
\big\|w^-_{\lambda} \big\|^2_{L^2(\Sigma_{\lambda}^-)}  .
\end{align*}
Since $u$ is exponential decay, we have $  \big\| (\lambda-y_1) |u|^{p-1}\big\|_{L^4(\Sigma_{\lambda}^-)} \longrightarrow 0  $ 
as $ \lambda \to -\infty,$  also for $ \big\|u^{p-1} \big\|_{L^4(\Sigma_{\lambda}^-)} . $ Choose $\lambda $ negative enough again, then we have
\begin{align*}
\int_{\Sigma_{\lambda}^-}  L_{\lambda}^+  |u_{\lambda}|^{p-2} u_{\lambda} w^-_{\lambda} \leq \frac{1}{8} \big\|w^-_{\lambda} \big\|^2_{L^2(\Sigma_{\lambda}^-)}.
\end{align*}
We estimate the $M_{\lambda}^+$ in the following way.
\begin{align*}
 M_{\lambda} &=  2 \int_{\Sigma_{\lambda}}     \bigg( \log \frac{|x_1-y_1|^2+|x_2+y_2|^2}{|x_1-y^{\lambda}_1|^2+ |x_2+y_2|^2 } \bigg)
 \bigg( |u_{\lambda}|^p -|u|^p   \bigg) dy   \\
 &= 2 \int_{\Sigma_{\lambda}^+}  \bigg( \log  \frac{|x_1-y_1|^2+|x_2+y_2|^2}{|x_1-y^{\lambda}_1|^2+ |x_2+y_2|^2 } \bigg) \bigg( |u_{\lambda}|^p -|u|^p   \bigg) dy \\
  &\ \ \  +  2 \int_{\Sigma_{\lambda}^-}  \bigg( \log  \frac{|x_1-y_1|^2+|x_2+y_2|^2}{|x_1-y^{\lambda}_1|^2+ |x_2+y_2|^2 } \bigg) \bigg( |u_{\lambda}|^p -|u|^p   \bigg) dy ;  \\
M_{\lambda}^+    &\leq  2 \int_{\Sigma_{\lambda}^-}  \bigg( \log  \frac{|x_1-y_1|^2+|x_2+y_2|^2}{|x_1-y^{\lambda}_1|^2+ |x_2+y_2|^2 } \bigg) \bigg( |u_{\lambda}|^p -|u|^p   \bigg) dy     \\
  &=  2 \int_{\Sigma_{\lambda}^-}  \bigg( \log  \frac{|x_1-y_1^{\lambda}|^2+|x_2+y_2|^2}{|x_1-y_1|^2+ |x_2+y_2|^2 } \bigg) \bigg( 
   |u|^p -|u_{\lambda}|^p  \bigg) dy   \\
  &\leq C   \int_{\Sigma_{\lambda}^-}  \bigg( \log  \frac{|x_1-y_1^{\lambda}|^2+|x_2+y_2|^2}{|x_1-y_1|^2+ |x_2+y_2|^2 } \bigg) 
  \bigg( |\eta_{\lambda}|^{p-2}\eta_{\lambda} w^-_{\lambda} \bigg) dy   \\
  &\leq C  \int_{\Sigma_{\lambda}^-}  \bigg( \log \big( 1 + \frac{|x_1-y_1^{\lambda}|^2 - |x_1-y_1|^2}{|x_1-y_1|^2+ |x_2+y_2|^2 } \big)\bigg) 
  \bigg( |\eta_{\lambda}|^{p-2}\eta_{\lambda} w^-_{\lambda} \bigg) dy  \\
  &\leq C  \int_{\Sigma_{\lambda}^-}  \bigg( \log \big( 1 + \frac{|x_1-y_1^{\lambda}|^2 - |x_1-y_1|^2}{|x-y|^2 } \big)\bigg) 
  \bigg( |\eta_{\lambda}|^{p-2}\eta_{\lambda} w^-_{\lambda} \bigg) dy \\
  &\leq C  \int_{\Sigma_{\lambda}^-}  \bigg( \log \big( 1 + \frac{\sigma |x_1-y_1|^2 + C_{\sigma}|y_1^{\lambda}-y_1 |^2}{|x-y|^2 } \big)\bigg) 
  \bigg( |\eta_{\lambda}|^{p-2}\eta_{\lambda} w^-_{\lambda} \bigg) dy \\
  &\leq C \int_{\Sigma_{\lambda}^-}   \frac{ \big(\sigma |x_1-y_1|^2 + C_{\sigma} |y_1^{\lambda}-y_1 |^2|  \big) ^{\frac{1}{2}}}{|x-y| }
  \bigg( |\eta_{\lambda}|^{p-2}\eta_{\lambda} w^-_{\lambda} \bigg) dy \\
  &\leq C \int_{\Sigma_{\lambda}^-}   \frac{  \sigma^{\frac{1}{2}} |x_1-y_1| + C_{\sigma} |y_1^{\lambda}-y_1 | }{|x-y| }
    \bigg( |\eta_{\lambda}|^{p-2}\eta_{\lambda} w^-_{\lambda} \bigg) dy \\
   &\leq C \sigma^{\frac{1}{2}}  \int_{\Sigma_{\lambda}^-} |u|^{p-1}  w^-_{\lambda}dy + C_{\sigma} \int_{\Sigma_{\lambda}^-}  
   \frac{1}{|x-y|}  (\lambda-y_1)|u|^{p-1}  w^-_{\lambda} dy\\
   &\leq C \sigma^{\frac{1}{2}}  \big\|u^{p-1} \big\|_{L^2(\Sigma_{\lambda}^-)} 
   \big\|w^-_{\lambda} \big\|_{L^2(\Sigma_{\lambda}^-)}   + C_{\sigma} N_{\lambda}(x).
\end{align*}
So we have the estimates for the second term in \eqref{eq4.4}:
\begin{align*}
 \int_{\Sigma_{\lambda}^-}  M_{\lambda}^+    |u_{\lambda}|^{p-2} u_{\lambda} w^-_{\lambda}
 &\leq C \sigma^{\frac{1}{2}} \big\|u^{p-1} \big\|_{L^2(\Sigma_{\lambda}^-)} 
 \big\|w^-_{\lambda} \big\|_{L^2(\Sigma_{\lambda}^-)}   \int_{\Sigma_{\lambda}^-} |u|^{p-1}  w^-_{\lambda}dy \\
 &\ \ \  +C_{\sigma}   \int_{\Sigma_{\lambda}^-} N_{\lambda}  |u|^{p-1}  w^-_{\lambda}dy \\
  &\leq C \sigma^{\frac{1}{2}} \big\|u^{p-1} \big\|^2_{L^2(\Sigma_{\lambda}^-)} \big\|w^-_{\lambda} \big\|^2_{L^2(\Sigma_{\lambda}^-)} \\
  &\ \ \  + C_{\sigma}  \big\|N_{\lambda} \big\|_{L^4(\Sigma_{\lambda}^-)}   \big\|u^{p-1} \big\|_{L^4(\Sigma_{\lambda}^-)} 
  \big\|w^-_{\lambda} \big\|_{L^2(\Sigma_{\lambda}^-)} \\
  &\leq C \sigma^{\frac{1}{2}} \big\|u^{p-1} \big\|^2_{L^2(\Sigma_{\lambda}^-)} \big\|w^-_{\lambda} \big\|^2_{L^2(\Sigma_{\lambda}^-)} \\
  &\ \ \  + C_{\sigma} \big\| (\lambda-y_1) |u|^{p-1}\big\|_{L^4(\Sigma_{\lambda}^-)} \big\|u^{p-1} \big\|_{L^4(\Sigma_{\lambda}^-)} 
  \big\|w^-_{\lambda} \big\|^2_{L^2(\Sigma_{\lambda}^-)} .
\end{align*}
Now again taking $\sigma \to 0,$ $\lambda \to -\infty,$ we get 
$$ \int_{\Sigma_{\lambda}^-}  M_{\lambda}^+    |u_{\lambda}|^{p-2} u_{\lambda} w^-_{\lambda} \leq \frac{1}{8}  
\big\|w^-_{\lambda} \big\|^2_{L^2(\Sigma_{\lambda}^-)} .$$
Combining these two estimates, we arrive at
$$  \int_{\Sigma_{\lambda}^-} |\nabla w^-_{\lambda} |^2  + \int_{\Sigma_{\lambda}^-} |w^-_{\lambda}| ^2 \leq \frac{1}{4}
 \big\|w^-_{\lambda} \big\|^2_{L^2(\Sigma_{\lambda}^-)}.  $$
 It follows $  \int_{\Sigma_{\lambda}^-}   |w^-_{\lambda}| ^2 =0.   $ Then the set $ \Sigma_{\lambda}^- $ is of measure zero:
  $ \mathcal{L}^n (\Sigma_{\lambda}^-) =0 .$  Since if it has positive measure, say $ \mathcal{L}^n (\Sigma_{\lambda}^-)\geq 2 \delta, $ 
  then we can choose a compact subset $K \subset \Sigma_{\lambda}^-$ such that $  \mathcal{L}^n (K) \geq \delta  .$ 
  Notice the definition of the set $ \Sigma_{\lambda}^-  .$  We get $ w^-_{\lambda} $ is positive in $ \Sigma_{\lambda}^- $, hence has a positive 
  lower bound on $K$ by continuity, say $ w^-_{\lambda} \geq \kappa >0. $  Then $ \int_{\Sigma_{\lambda}^-}   |w^-_{\lambda}| ^2 \geq \kappa^2 \delta>0. $
  Then by the continuity of $w^-_{\lambda},$  we get $ \Sigma_{\lambda}^- = \emptyset $ from $ \mathcal{L}^n (\Sigma_{\lambda}^-) =0 .$
 In fact, we can use  $ \int_{\Sigma_{\lambda}^-} |\nabla w^-_{\lambda} |^2 =0 $ to derive $ w^-_{\lambda} =0.$ 
 Anyway we have $ \Sigma_{\lambda}^- = \emptyset $ and $ w_{\lambda}=u_{\lambda} -u \geq 0 $ for $\lambda$ negative enough.\\
 \textbf{Step 2:} Move the plane to the limiting position.\\
 \indent Define $\lambda_0:= \sup\big\{ \lambda | w_{\mu} \geq 0 \; \text{for all} \; \mu \leq \lambda \big\}.$ Then by the same argument as in Step 1 from 
 the right direction $x_1 =+\infty,$ we see $\lambda_0 < +\infty.$ Now, we prove $w_{\lambda_{0}} =u_{\lambda_{0}} -u=0 $ to get the axial symmetry.
 We show this by contracdiction. If not, we will prove there exists an $\epsilon >0$ small enough, such that for all
  $\lambda\in (\lambda_0,\lambda_0 +\epsilon),$
 we still have have $w_{\lambda} \geq 0,$ which will be contradicted with definition of $\lambda_0.$\\
\indent Suppose now $w_{\lambda_0} \not \equiv 0.$ Then $w_{\lambda_0} \geq 0$ and $ w_{\lambda_0}(x_0) >0$ for some $x_0 \in \Sigma_{\lambda_0}.$
 By the integral representation of $L_{\lambda_0}$ and $M_{\lambda_0},$ we see $ L_{\lambda_0}<0 ,$  $ M_{\lambda_0} <0 $ strictly. 
 Then by the Euler-Lagrange equation of $w_{\lambda_0}$, we have 
 $$  -\Delta w_{\lambda} + w_{\lambda} + \frac{1}{2}(p-1) H(x) |\psi_{\lambda}|^{p-2} w_{\lambda} =
- \frac{1}{2} L_{\lambda} |u_{\lambda}|^{p-2} u_{\lambda}   -  \frac{1}{2} M_{\lambda} |u_{\lambda}|^{p-2} u_{\lambda} >0.
    $$
 From this, by the maximun principle, we get $ w_{\lambda_0}>0 $ in $\Sigma_{\lambda_0},$ and $ \frac{\partial u}{\partial x_1} >0 $ along 
 the $-x_1$ direction of the ray.
 Now taking $R$ large enough such that $H(x)$ large enough in $\mathbb{R}^2_+ \setminus B_{R}(0).$
 For $\Sigma_{\lambda_0} \cap B_R, $ we have $ w_{\lambda_0} >0. $ Then by the continuity of $w_{\lambda}(x)=w(\lambda,x),$
 there is an $ \epsilon >0$ small enough, such that $w_{\lambda}\big|_{\Sigma_{\lambda}\cap B_R} >0 $ 
 for all $\lambda \in (\lambda_0,\lambda_0 + \epsilon),$ which give us $w^-_{\lambda} =0.$
 Then we estimate the integrals in \eqref{eq4.4} as in the Step 1, we have 
 \begin{align*}
 \int_{\Sigma_{\lambda}^-} |\nabla w^-_{\lambda} |^2  + \int_{\Sigma_{\lambda}^-} |w^-_{\lambda}| ^2  
 \leq \frac{1}{2} \int_{\Sigma_{\lambda}^- \cap B_R^c}  L_{\lambda}^+    |u_{\lambda}|^{p-2} u_{\lambda} w^-_{\lambda} +
 \frac{1}{2}  \int_{\Sigma_{\lambda}^- \cap B_R^c}  M_{\lambda}^+    |u_{\lambda}|^{p-2} u_{\lambda} w^-_{\lambda}.
 \end{align*}
 All the estimates are exactly the same as in Step 1, except the integrals are over $ \Sigma_{\lambda}^- \cap B_R^c .$
 Now taking $R \longrightarrow \infty$ in place of $\lambda \longrightarrow -\infty $ in Step 1, we  again get
 $$  \int_{\Sigma_{\lambda}^-} |\nabla w^-_{\lambda} |^2  + \int_{\Sigma_{\lambda}^-} |w^-_{\lambda}| ^2 \leq \frac{1}{4}
 \big\|w^-_{\lambda} \big\|^2_{L^2(\Sigma_{\lambda}^- \cap B^c_R)}.  $$
 From this, we get $\Sigma^-_{\lambda} = \emptyset$ hence $ w_{\lambda} =u_{\lambda} -u \geq 0 $ for all $ \lambda \in (\lambda_0,\lambda_0+\epsilon),$
 a contradiction with the definition of $\lambda_0.$ So we have $w_{\lambda_0} \equiv 0$ and $ \frac{\partial u}{\partial x_1} <0 $
 along the ray starting from the axis in $x_1$ direction. 
\end{proof}
\par\noindent
\par\noindent
{\bf{Acknowledgement:}} Y. Zhang was supported by the Postdoctoral Scientific Research Foundation of Central South University.
%%%%%%%%%%%%%%%%%%%%%%%%%%%%%%%%%%%%%%%%%%%%%%%%%%%%%%%%%%%%%%%%%%%%%%%%%%%%%%%%%%%%%%
%%%%%%%%%%%%%%%%%%%%%%%%%%%%%%%%%%%%%%%%%%%%%%%%%%%%%%%%%%%%%%%%%%%%%%%%%%%%%%%%%%%%%%%%%%%%%%%%%%%%%%%%%%%%%%%%%%%%%%%%%%%%%%%%%%%%%
%%%%%%%%%%%%%%%%%%%%%%%%%%%%%%%%%%%%%%%%%%%%%%%%%%%%%%%%%%%%%%%%%%%%%%%%%%%%%%%%%%%%%%%%%%%%%%%%%%%%%%%%%%%%%%%%%%%%%%%%%%%%%%%

%\printbibliography[title=References]


\begin{thebibliography}{99}
\bibitem{Acker1}  N. Ackermann,  {\it A nonlinear superposition principle and multibump solutions of periodic
Schr\"{o}dinger equations}, J. Funct. Anal. \textbf{234}(2006), no. 2, 277-320. 
 Math. Z., \textbf{248} (2004), no. 2,  423-443

\bibitem{Acker2}  N. Ackermann,  {\it On a periodic Schr\"{o}dinger equation with nonlocal superlinear part},  
 Math. Z., \textbf{248} (2004), no. 2,  423-443

\bibitem{BCS} D. Bonheure, S. Cingolani and J. V. Schaftingen, {\it The logarithmic Choquard equation: Sharp asymptotics and nondegeneracy of the groundstate}, J. Funct. Anal., \textbf{272} (2017), no. 12, 5255-5281.

\bibitem{B} Haim Brezis, {\it Functional analysis, Sobolev spaces and partial differential equations}, Universitext. Springer, New York, 2011.

\bibitem{BJ} J. Byeon and L. Jeanjean,  {\it Standing waves for nonlinear Schr\"{o}dinger equations
with a general nonlinearity},  Arch. Ration. Mech. Anal.,\textbf{185} (2007), no. 2, 185-200.

\bibitem{DZ} D. Cao, W. Dai and Y. Zhang,  {\it Existence and symmetry of positive solutions to 2-D Schr\"{o}dinger-Newton equations}, preprint, submitted for publication, 2019, 28pp

\bibitem{BS} L. Battaglia and J. V. Schaftingen, {\it Groundstates of the Choquard equations with a sign-changing self-interaction potential}, Z. Angew. Math. Phys., \textbf{69} (2018): 86.

\bibitem{CLO} W. Chen, C. Li and B. Ou, {\it Classification of solutions for an integral equation}, Comm. Pure Appl. Math., \textbf{59} (2006), 330-343.

\bibitem{CSV} P. Choquard, J. Stubbe and M. Vuffray, {\it Stationary solutions of the Schr\"{o}dinger-Newton model - an ODE approach}, Differential Integral Equations, \textbf{21} (2008), no. 7-8, 665-679.

\bibitem{CW} S. Cingolani and T. Weth, {\it On the planar Schr\"{o}dinger-Poisson system}, Ann. Inst. H. Poincar\'{e} Anal. Non Lin\'{e}aire, \textbf{33} (2016), no. 1, 169-197.

\bibitem{DQ} W. Dai and G. Qin, {\it Classification of nonnegative classical solutions to third-order equations}, 
Adv. Math., \textbf{328} (2018), 822-857.

\bibitem{DW} M. Du and T. Weth, {\it Ground states and high energy solutions of the planar Schr\"{o}dinger-Poisson system}, Nonlinearity, \textbf{30} (2017), no. 9, 3492-3515.

\bibitem{GS} M. Ghimenti and J. Van Schaftingen, {\it Nodal solutions for the Choquard equation},
J. Funct. Anal., \textbf{271} (2016), no. 1,  107-135.

\bibitem{GNN} B. Gidas, W. Ni and L. Nirenberg, {\it Symmetry and related properties via maximum principle}, Comm. Math. Phys., \textbf{68} (1979), 209-243.

\bibitem{LW} G. Li and C. Wang, {\it The existence of a nontrivial solution to a nonlinear
elliptic problem of linking type without the Ambrosetti-Rabinowitz condition}, 
 Ann. Acad. Sci. Fenn. Math., \textbf{ 36} (2011), no. 2,  461-480.

\bibitem{LiZhu} Y. Li and M. Zhu. {\it Uniqueness theorems through the method of moving
spheres},  Duke Math. J., \textbf{ 80}(1995), no. 2,  383-417.

\bibitem{L} E. H. Lieb, {\it Existence and uniqueness of the minimizing solution of Choquard's nonlinear equation}, Studies in Appl. Math., \textbf{57}(1976/77),  no. 2, 93-105.

\bibitem{LL} E. H. Lieb and M. Loss, {\it Analysis}, Second edition, Graduate Studies in Mathematics, \textbf{14}, American Mathematical Society, Providence, RI, 2001.

\bibitem{Lions} P.L. Lions,  {\it The Choquard equation and related questions},  Nonlinear Anal.,
\textbf{4} (1980), no. 6, 1063-1072.

\bibitem{MZ} L. Ma and L. Zhao, {\it Classification of positive solitary solutions of the nonlinear Choquard equation}, Arch. Rational Mech. Anal., \textbf{195} (2010), no. 2, 455-467.

\bibitem{MS1} V. Moroz and J. Van Schaftingen, {\it Ground states of nonlinear Choquard equations: existence, qualitative properties and decay asymptotics}, J. Funct. Anal., \textbf{265} (2013), no. 2, 153-184.

\bibitem{MS2} V. Moroz and J. Van Schaftingen, {\it A guide to the Choquard equation }, J. Fixed Point Theory Appl., \textbf{19} (2017), no. 1, 773-813.

\bibitem{Palais} R.S. Palais,  {\it The principle of symmetric criticality},Commun. Math.Phys., \textbf{69} (1979), 19-30.

\bibitem{Pekar} S. I. Pekar, {\it Untersuchungen \"{u}ber die Elektronentheorie der Kristalle}, Akademie-Verlag, Berlin, 1954.

\bibitem{P} R. Penrose, {\it On gravity's role in quantum state reduction}, Gen. Relativ. Gravit., \textbf{28} (1996), no. 5, 581-600.

\bibitem{R} D. Ruiz, {\it The Schr\"{o}dinger-Poisson equation under the effect of a nonlinear local term}, J. Funct. Anal., \textbf{137} (2006), 655-674.

\bibitem{S} M. Struwe, {\it Variational methods. Applications to nonlinear partial differential equations and Hamiltonian systems}, Second edition, Springer-Verlag, Berlin, 1996.

\bibitem{Stubbe} J. Stubbe, {\it Bound states of two-dimensional Schr\"{o}dinger-Newton equations}, arXiv:0807.4059, 2008.

\bibitem{WX}   J. Wei and X.Xu. {\it Classification of solutions of higher order conformally
invariant equations}, Math. Ann.,  \textbf{313}(1999), no,. 2,  207-228.

\bibitem{W} Michel Willem, {\it Minimax Theorems}, Birkh\"{a}user Boston, 1996.
\end{thebibliography}
\end{document}